\documentclass[11pt]{amsart}

\usepackage{amsmath, amsthm, amssymb,latexsym, graphicx, mathrsfs, manfnt, newtxtext, appendix, chemarrow, enumerate, enumitem, multicol, cite, soul, mathtools, kantlipsum, bm}
\allowdisplaybreaks
\usepackage[misc,geometry]{ifsym}
\usepackage[colorlinks=true,linkcolor=blue,citecolor=blue]{hyperref}
\usepackage[all]{xy}
\usepackage[utf8]{inputenc}
\usepackage[T1]{fontenc}

\numberwithin{equation}{section}
\allowdisplaybreaks
\theoremstyle{thmit} 
\newtheorem{theorem}{Theorem}[section]
\newtheorem{lemma}[theorem]{Lemma}
\newtheorem{corollary}[theorem]{Corollary}
\newtheorem{proposition}[theorem]{Proposition}

\newcommand{\snorm}[1]{{\left\vert #1 \right\vert}}

\newcommand{\norm}[1]{{\left\vert\kern-0.25ex\left\vert #1 
    \right\vert\kern-0.25ex\right\vert}}
\newcommand{\wbnorm}[3]{\norm{#1}_{#2,#3}}

\newcommand{\tnorm}[1]{{\left\vert\kern-0.25ex\left\vert\kern-0.25ex\left\vert #1 
    \right\vert\kern-0.25ex\right\vert\kern-0.25ex\right\vert}}
\newcommand{\ibnorm}[4]{\tnorm{#1}_{#2,#3,#4}}    

\newcommand\ces{\mathsf{C}}
\newcommand\ices{\mathsf{C}^{-1}}

\newcommand\udisk{\mathbb{D}}
\newcommand\pudisk{\mathbb{D}\setminus \{0\}}

\newcommand\dint{\int_{\mathbb{D}}}

\newcommand\intx[1]
{{\int\limits_{#1}}}

\newcommand\cn{\mathbb{C}}

\newcommand\nn{\mathbb{N}}

\newcommand\clo[1]{{\mathcal{L}(#1)}}

\newcommand{\ptsp}[2]{\sigma_{pt}(#1;#2)}
\newcommand{\spec}[2]{\sigma(#1;#2)}
\newcommand{\stsp}[2]{\sigma^*(#1;#2)}

\newcommand\ind{\operatorname{ind}}
\newcommand\indx[2]{{\underset{#1 \in #2}{\operatorname{ind}}}}

\newcommand\projx[2]{{\underset{#1 \in #2}{\operatorname{proj}}}}
\newcommand{\goesto}[1]{\xrightarrow[#1]{}}
\newcommand\supx[2]{{\sup_{#1 \in #2}}}

\newcommand\limx[2]{{\lim_{#1 \to #2}}}
\newcommand\sumx[3]{{\sum_{#1=#2}^{#3} }}

\newcommand{\tay}[2]{{\hat{#1}(#2)}}

\newcommand{\iberg}[2]{A^{#1}_{#2 +}}
\newcommand{\uberg}[2]{A^{#1}_{#2 -}}

\newcommand\holom{H(\mathbb{D})}
\newcommand\analy{H^\infty}
\newcommand{\wberg}[2]{A^{#1}_{#2}}
\newcommand{\imsp}[1]
{{\operatorname{Im}(\lambda I-#1)}}

\newcommand\dif{\mathrm{d}}
\newcommand{\difs}[1]{{\mathrm{d}s(#1)}}
\newcommand{\diff}[1]{{\mathrm{d}s_{#1}(z)}}

\newcommand{\alphapp}[1]
{{(\alpha+\frac{1}{#1})}}

\newcommand{\alphapm}[1]
{{(\alpha-\frac{1}{#1})}}

\begin{document}

\author{Ersin Kızgut}

\address{Instituto Universitario de Matemática Pura y Aplicada (IUMPA)\\ Universitat Politècnica de València \\ 46071 Valencia \\ Spain \\
}
\email{erkiz@upv.es}

\title{The Cesàro operator on weighted Bergman Fréchet and (LB)-spaces of analytic functions}

\maketitle

\begin{abstract}
	The spectrum of the Cesàro operator $\mathsf{C}$ is determined on the spaces which arises as intersections $A^p_{\alpha +}$ (resp. unions $A^p_{\alpha -}$) of Bergman spaces $A_\alpha^p$ of order $1<p<\infty$ induced by standard radial weights $(1-|z|)^\alpha$, for $0<\alpha<\infty$. We treat them as reduced projective limits (resp. inductive limits) of weighted Bergman spaces $A^p_\alpha$, with respect to $\alpha$. Proving that these spaces admit the monomials as a Schauder basis paves the way for using Grothendieck-Pietsch criterion to deduce that we end up with a non-nuclear Fréchet-Schwartz space (resp. a non-nuclear (DFS)-space). We show that $\ces$ is always continuous, while it fails to be compact or to have bounded inverse on $A^p_{\alpha +}$ and $A^p_{\alpha -}$.
\end{abstract}

\section{Introduction}
Let $\holom$ denote the Fréchet space of all analytic functions $f\colon \udisk \to \cn$ equipped with the topology of uniform convergence on the compact subsets of the unit disc $\udisk:=\{z \in \cn:\snorm{z}<1\}$. The classical Cesàro operator $\ces$ is given by
\begin{equation}\label{ces def}
	f \mapsto \ces(f): z \mapsto \frac{1}{z} \int_0^z \frac{f(\zeta)}{1-\zeta} \dif \zeta, \quad z \in \pudisk, \, \ces(f)(0):=f(0),
\end{equation}
for $f \in \holom$. The Cesàro operator $\ces$ is an isomorphism of $\holom$ onto itself. From \eqref{ces def} one may obtain that
\begin{equation}\label{recover from ces}
	f(z)= (1-z)(z\ces(f)(z))', \quad f \in \holom.
\end{equation}
In the sense of Taylor coefficients
\begin{equation}\label{tayf}
\tay{f}{j} :=\frac{f^{(j)}(0)}{j!}, \quad n \in \nn_0
\end{equation}
of the function $f\in \holom$ given by
\begin{equation}\label{taylor representation}
	f(z)=\sumx{j}{0}{\infty} \tay{f}{j} z^j,
\end{equation}
one has the expression
\begin{equation}
	\ces f(z)=\sumx{k}{0}{\infty} \left(\frac{1}{k+1}\sumx{j}{0}{k} \tay{f}{j}\right)z^k, \quad z \in \udisk.
\end{equation}
 Let $\clo{E}$ denote the space of all continuous functions on a topological vector space $E$. For $\varphi \in \analy$, we denote $M_\varphi$ the operator in $\holom$ of multiplication by $\varphi$ so that $M_\varphi \in \clo{\holom}$. The differentiation operator $D \colon f \mapsto f'$ for $f \in \holom$ also belongs to $\clo{\holom}$. Then it follows from \eqref{ces def} that $\ices \in \clo{\holom}$ given by $\ices=M_{1-z}DM_z$ (see \cite[p. 1185]{Per08}). That is, for all $h \in \holom$ we have
\begin{equation}\label{ices with mult}
	\ices (h)(z)=(1-z)(h(z)+zh'(z)), \quad z \in \udisk.
\end{equation}
The continuity, compactness and spectrum of generalized Cesàro operators on Banach spaces of analytic functions on the unit disc have been studied by many authors \cite{AMN05, AC01, Ale09, Ale10, Ale12, AS95, AS97, And96, Cow84, CR11, HLP34, Per08, Sis90, Sis96}. Continuity of the Cesàro operator on the Hilbert space $H^2(\udisk)$ was studied by Hardy, Littlewood, and Pólya \cite{HLP34}. Continuity of $\ces$ on the general Hardy spaces and unweighted Bergman spaces $A^p$ is due to Siskakis \cite{Sis90,Sis96}. Andersen \cite{And96} proved that the Cesàro operator is bounded on a class of spaces of analytic functions on the unit disc, including the weighted Bergman space. Boundedness and compactness of the class of a certain type of integral operators (also containing the Cesàro operator) acting on spaces of analytic functions on the unit disc have been studied by Aleman and Cima \cite{AC01}, Aleman and Siskakis \cite{AS97}. Persson \cite{Per08}, Aleman and Constantin \cite{Ale09}, Aleman and Persson \cite{Ale10} investigated the spectrum of (generalized) Cesàro operators on various spaces of analytic functions such as Hardy spaces, weighted Bergman spaces and Dirichlet spaces in detail. We refer the reader to the introduction of \cite{Ale10} for a comprehensive information on the development of the research in this area. The \textit{Bergman space} $\wberg{p}{\alpha} =\wberg{p}{\alpha}(\udisk)$ of order $1<p<\infty$ induced by standard radial weight $(1-\snorm{z})^\alpha$ for $0<\alpha<\infty$ is given by
\begin{equation}\label{weighted bergman space}
	\wberg{p}{\alpha}:=\{f \in \holom: \wbnorm{f}{p}{\alpha} =\left(\dint \snorm{f(z)}^p \diff{\alpha}\right)^{1/p}<\infty\},
\end{equation}
where $\diff{\alpha}=(1-\snorm{z})^\alpha \difs{z}$, and $\difs{z}=\frac{1}{\pi}\dif x \dif y$. Some authors prefer to define the space $\wberg{p}{\alpha}$ with the weight $(1-\snorm{z}^2)^\alpha$ instead of $(1-\snorm{z})^\alpha$. Since we have $1-\snorm{z} \leq 1-\snorm{z}^2 \leq 2(1-\snorm{z})$, these spaces coincide and the norms are equivalent. Each $\wberg{p}{\alpha}$ is a closed subspace of $L^p(\udisk, \diff{\alpha})$ in which the polynomials are dense \cite[Section 1.1]{HKZ00}. The weighted Bergman space $\wberg{p}{\alpha}$ is a Banach space with the norm $\wbnorm{\cdot}{p}{\alpha}$. Classical Bergman space $A^p(\udisk)$ corresponds to the case $\alpha=0$. Contrary to $\holom$, the Cesàro operator $\ces$ does not have a bounded inverse on the Banach space $\wberg{p}{\alpha}$ (see e.g. \cite{Per08}). The aim of this paper is to investigate the Cesàro operator $\ces$ on spaces that arise as intersections and unions of Bergman spaces of order $1<p<\infty$ induced by the standard weights $(1-\snorm{z})^\alpha$ for $0<\alpha<\infty$:
\begin{align}
	\iberg{p}{\alpha} & :=\{f \in \holom:\left(\dint \snorm{f(z)}^p \diff{\mu} \right)^{1/p}<\infty, \, \forall \mu>\alpha\} \nonumber \\
	& = \bigcap_{\mu>\alpha} \wberg{p}{\mu}  = \bigcap_{n \in \nn} \wberg{p}{\alphapp{n}} \nonumber \\ 
	& =  \projx{n}{\nn} \wberg{p}{\alphapp{n}}, \label{projlim verion} \\
	\uberg{p}{\alpha} & := \{f \in \holom:\left(\dint \snorm{f(z)}^p \diff{\mu} \right)^{1/p} <\infty, \text{ for some }\mu <\alpha\}\nonumber \\
	& = \bigcup_{\mu<\alpha} \wberg{p}{\mu} = \bigcup_{n \in \nn} \wberg{p}{\alphapm{n}}\nonumber\\ 
	& =\indx{n}{\nn} \wberg{p}{\alphapm{n}} \label{injlim verion},
\end{align}
The monograph \cite{HKZ00} presents an investigation of Bergman type spaces (of infinite order) in that fashion with relevance to interpolation and sampling of analytic functions. The paper \cite{KR05} gives a description of intersections and unions of weighted Bergman spaces of order $0<p<\infty$. Unlike those, we treat the space $\iberg{p}{\alpha}$ as a Fréchet space when equipped with the locally convex topology generated by the increasing system of norms
\begin{equation}\label{norm in ibergp}
\ibnorm{f}{p}{\alpha}{n}	:=\left(\dint \snorm{f(z)}^p \diff{\alphapp{n}}\right)^{1/p},
\end{equation}
for $f \in \iberg{p}{\alpha}$ and each $n \in \nn$. The space $\uberg{p}{\alpha}$ is an (LB)-space endowed with the finest locally convex topology such that each natural inclusion map from $\wberg{p}{\mu}$ into $\wberg{p}{\gamma}$, for $0<\mu<\gamma$ is continuous. It is also regular, since every bounded set $B \subseteq \uberg{p}{\alpha}$ is contained and bounded in the Banach space $\wberg{p}{\mu}$, for some $0<\mu<\alpha$. We also mention that for $0<\beta<\alpha<\infty$, we have $\wberg{p}{\beta} \subset \uberg{p}{\alpha} \subset \wberg{p}{\alpha} \subset \iberg{p}{\alpha}$.

We address the inspiration and motivation of this research to three sources. Aleman and Constantin \cite{Ale09} investigated the spectrum of the Cesàro operator on weighted Bergman spaces. Albanese, Bonet and Ricker \cite{ABR16} studied continuity, compactness and spectrum of the Cesàro operator in growth Banach spaces. In \cite{ABR18-3} they conducted the same investigation for Cesàro operator within the context of intersections and unions of growth spaces. We keep their setup in present paper, and use it for weighted Bergman spaces. Retaining similar techniques, the patterns of our proofs are very close to the those of \cite{ABR18-3} concerning the spectrum of the Cesàro operator. 

The paper is organized as follows. In Section~\ref{section: ces on iberg}, we focus on the properties of the Cesàro operator $\ces$ defined on $\iberg{p}{\alpha}$ and $\uberg{p}{\alpha}$. We immediately reach that $\ces$ is always continuous on $\iberg{p}{\alpha}$ and $\uberg{p}{\alpha}$ since it is continuous on each step. To show that $\ces$ is not an isomorphism on $\iberg{p}{\alpha}$ and $\uberg{p}{\alpha}$, we construct specific functions. Then, we determine its spectrum on $\iberg{p}{\alpha}$ and $\uberg{p}{\alpha}$. Spectral properties of $\ces$ reveals that it is non-compact on $\iberg{p}{\alpha}$ and $\uberg{p}{\alpha}$. In Section~\ref{section: discussion on iberg} we concentrate on the structural properties of $\iberg{p}{\alpha}$ and $\uberg{p}{\alpha}$. By proving that for each pair $0<\mu<\gamma<\infty$, each inclusion map from $\wberg{p}{\mu}$ to $\wberg{p}{\gamma}$ is compact, we establish that $\iberg{p}{\alpha}$ is a Fréchet-Schwartz space and $\uberg{p}{\alpha}$ is a (DFS)-space. Using the approach in \cite[Section 2]{BLT18}, we show that $\iberg{p}{\alpha}$ and $\uberg{p}{\alpha}$ are non-nuclear spaces both admitting the monomials $\{z^j\}_{j=1}^\infty$ as a Schauder basis. 

\section{The Cesàro operator $\ces$ on $\iberg{p}{\alpha}$ and $\uberg{p}{\alpha}$}\label{section: ces on iberg}
\subsection{$\ces$ is continuous on $\iberg{p}{\alpha}$ and $\uberg{p}{\alpha}$}
Boundedness of the Cesàro operator on the Banach space $\wberg{p}{\alpha}$ is due to Andersen \cite[Corollary 1.2]{And96}, as it is proved for a more general class of spaces of analytic functions on the unit disc. This implies $\ces$ is continuous at every step $\wberg{p}{\alphapp{n}}$ or $\wberg{p}{\alphapm{n}}$, and hence $\ces$ is also continuous on $\iberg{p}{\alpha}$ and $\uberg{p}{\alpha}$ by means of the properties of Fréchet and inductive limit topologies, respectively. Both $\iberg{p}{\alpha}$ and $\uberg{p}{\alpha}$ contain polynomials. Hence $\iberg{p}{\alpha}$ is dense in $\wberg{p}{\mu}$, for every $\mu>\alpha$ and $\uberg{p}{\alpha}$ is dense in $\wberg{p}{\alpha}$, for every $0<\alpha<\infty$. Let us denote $\iota_n\colon\iberg{p}{\alpha} \hookrightarrow\wberg{p}{\alphapp{n}}$ and $\iota_{n,n+1}\colon \wberg{p}{\alphapp{n+1}}\hookrightarrow \wberg{p}{\alphapp{n}}$ the canonical maps with dense range.  We denote the Cesàro operator $\ces_n \colon \wberg{p}{\alphapp{n}} \to \wberg{p}{\alphapp{n}}$, for each $n \in \nn$. Observe that $\iota  \ces = \ces_n  \iota_n$ and $\iota_{n,n+1} \ces_{n+1}=\ces_n \iota_{n,n+1}$, for every $n \in \nn$. 

\subsection{Inverse of $\ces$ on $\iberg{p}{\alpha}$ and $\uberg{p}{\alpha}$}
\begin{proposition}\label{ces isom on ibergp}
	Let $1<p<\infty$, and $0 < \alpha<\infty$. Then, 
\begin{enumerate}[label=\normalfont(\arabic*)]
	\item The Cesàro operator $\ces$ fails to be an isomorphism on $\iberg{p}{\alpha}$.
	\item The Cesàro operator $\ces$ fails to be an isomorphism on $\uberg{p}{\alpha}$.
\end{enumerate}
	
\end{proposition}

\begin{proof}
	Let $p>1$. Then, there exists $\varepsilon\in (0,1)$ such that $p\geq 1+2\varepsilon$. 
\begin{enumerate}[wide, labelwidth=!, labelindent=0pt, label=\normalfont(\arabic*)]\setlength\itemsep{0.5em}
\item Given $0<\alpha<\infty$, define $f_\varepsilon(z):=\frac{1}{(1+z)^{\frac{\alpha+1-\varepsilon}{p}}}$ for $z \in \udisk$. Since $1=\snorm{1+z-z}\leq \snorm{1+z}+\snorm{z}$, straightforward calculation shows $f_\varepsilon \in \wberg{p}{\alpha}$, and so $f_\varepsilon \in \iberg{p}{\alpha}$. Suppose that $\ices f_\varepsilon \in \iberg{p}{\alpha}$. Since clearly $(1-z)f_\varepsilon \in \iberg{p}{\alpha}$, this is equivalent to assume that $z(1-z)f'_\varepsilon \in \iberg{p}{\alpha}$. Let us define the region $A \subset \udisk$  by the intersection of $\operatorname{Re}(z)\leq -\frac{1}{2}$, and the Stolz angle with vertex $(-1,0)$ in which the inequality $\snorm{1+z}\leq 2 (1-\snorm{z})$ is satisfied (see e.g. \cite[Lemma 6.20]{Loy18}). It is easy to verify that $\snorm{1+z}<1$ and $\snorm{1-z}>\snorm{z}>\frac{1}{2}$ whenever $z \in A$. Let us pick $n_0 \in \nn$ such that $\frac{1}{n}<\varepsilon$, for all $n \geq n_0$.  Hence, $p+1-\varepsilon-\frac{1}{n} \geq 2$ for every $n\geq n_0$. Then, by the fact that $z(1-z)f'_\varepsilon \in \wberg{p}{\alphapp{n}}$, for a constant $M>0$ and for all $n\geq n_0$ we have
\begin{align*}
	M & \geq \ibnorm{z(1-z)f'_\varepsilon}{p}{\alpha}{n}^p = \dint \snorm{z(1-z)f'_\varepsilon(z)}^p \diff{\alphapp{n}} \\
	& \geq C(\alpha,\varepsilon)\intx{A} \snorm{\frac{z(1-z)}{(1+z)^{\frac{\alpha+1-\varepsilon}{p}+1}}}^p \diff{\alphapp{n}} \\
	&\geq C(p,\alpha,\varepsilon)  \intx{A} \snorm{\frac{1}{(1+z)^{\frac{\alpha+1-\varepsilon}{p}+1}}}^p\diff{\alphapp{n}} \\
	& \geq C(p,\alpha,n,\varepsilon) \intx{A} \frac{\difs{z}}{\snorm{1+z}^2}  = C(p,\alpha,n,\varepsilon) \int\limits_{-1}^{-\frac{1}{2}} \frac{1}{1+x} \arctan\left(\frac{y_0}{1+x}\right) \dif x,
\end{align*}
where $y_0:=\frac{1}{3}\sqrt{-9x^2+6x-8\sqrt{6x+7}+23}$. An integration by parts shows that the right hand side fails to be convergent. This is a contradiction.
\item
Define $g_\varepsilon:=\frac{1}{(1+z)^{\frac{\alpha+1-2\varepsilon}{p}}}$, for $z \in \udisk$. Note that there exists $n_\varepsilon \in \nn$ such that $\frac{1}{n}<\varepsilon$, for all $n \geq n_\varepsilon$. Then, similar to part (1) we obtain  $g_\varepsilon \in \wberg{p}{\alphapm{n_\varepsilon}}$, and so $g_\varepsilon \in \uberg{p}{\alpha}$. Now suppose that $\ices$ is continuous on $\uberg{p}{\alpha}$. Then, there exists $m>n_\varepsilon$ such that the restriction $\ices\colon\wberg{p}{\alphapm{n_\varepsilon}}\to\wberg{p}{\alphapm{m}}$ is continuous. Then, for a constant $M>0$, and the region $A\subset \udisk$ defined in part (1)
\begin{align*}
	M & \geq \wbnorm{z(1-z)g'_\varepsilon}{p}{\alphapm{m}}^p  = \dint \snorm{z(1-z)g'_\varepsilon(z)}^p \diff{\alphapm{m}}\\
	& \geq C(\alpha,\varepsilon) \intx{A} \snorm{\frac{z(1-z)}{(1+z)^{\frac{\alpha+1-2\varepsilon}{p}+1}}}^p \diff{\alphapm{m}}\\
	& \geq C(p,\alpha,\varepsilon) \intx{A} \frac{\diff{\alphapm{m}}}{\snorm{1+z}^{\alpha+1-2\varepsilon+p}} \\
	& \geq C(p,\alpha,m,\varepsilon) \intx{A} \frac{\difs{z}}{\snorm{1+z}^{p+1-2\varepsilon+\frac{1}{m}}} \geq C(p,\alpha,m,\varepsilon) \intx{A} \frac{\difs{z}}{\snorm{1+z}^2}.
\end{align*}
Similar to part (1), the right hand side is not convergent. This is a contradiction. 
\end{enumerate}
\end{proof}

\begin{corollary}
	Let $1<p<\infty$, and $0 <\alpha<\infty$. Then, 
\begin{enumerate}[label=\normalfont(\arabic*)]
	\item The differentiation operator $D$ does not act on on $\iberg{p}{\alpha}$.
	\item The differentiation operator $D$ does not act on on $\uberg{p}{\alpha}$.
\end{enumerate}
\end{corollary}
\begin{proof}
	Suppose that $D \in \clo{\iberg{p}{\alpha}}$. Since clearly both $M_{1-z}$ and $M_z$ are continuous in $\iberg{p}{\alpha}$, then by \eqref{ices with mult}, $\ices=M_{1-z} D  M_z$ is continuous in $\iberg{p}{\alpha}$. However, this contradicts Proposition~\ref{ces isom on ibergp}(1). Part (2) is similar. 
\end{proof}

\subsection{Spectrum of $\ces$ on $\iberg{p}{\alpha}$}\label{subsection: spectrum on i}
Let $X$ be a locally convex Hausdorff space, and $\Gamma_X$ a system of continuous seminorms determining the topology of $X$. Let $X'$ denote the space of all continuous linear functionals on $X$. Denote the identity operator on $X$ by $I$. Let $\clo{X}$ denote the space of all continuous linear operators from $X$ into itself. For $T \in \clo{X}$, the \textit{resolvent set} $\rho(T)$ of $T$ consists of all $\lambda \in \cn$ such that $R(\lambda,T):=(\lambda I-T)^{-1}$ exists in $\clo{X}$. The set $\spec{T}{X}:=\cn\setminus \rho(T)$ is called the \textit{spectrum} of $T$. The \textit{point spectrum} $\ptsp{T}{X}$ of $T$ consists of all $\lambda \in \cn$ such that $(\lambda I-T)$ is not injective. Contrary to Banach spaces, concerning the spectrum of an operator $T$ on the Fréchet space $X$, one may encounter that $\rho(T) = \varnothing$ or $\rho(T)$ fails to be an open set in $\cn$. For this reason, some authors prefer to consider the subset $\rho^*(T)$ of $\rho(T)$ consisting of $\lambda \in \cn$ such that there exists $\delta>0$ such that $B(\lambda,\delta):=\{z \in \cn: \snorm{z-\lambda}<\delta\}\subseteq \rho(T)$ and $\{R(\mu,T):\mu \in B(\lambda,\delta)\}$ is equicontinuous in $\clo{X}$. Define the \textit{Waelbroeck spectrum} $\stsp{T}{X}:= \cn\setminus \rho^*(T;X)$, which is a closed set containing $\spec{T}{X}$. If $T \in \clo{X}$ with $X$ a Banach space, then $\stsp{T}{X}=\spec{T}{X}$. For every $r\geq 1$ we denote the open disk by $D_r:=\left\{\lambda \in \cn:\snorm{\lambda-\frac{1}{2r}}<\frac{1}{2r}\right\}$. Let us write $\overline{D}_r:=\left\{\lambda \in \cn:\snorm{\lambda-\frac{1}{2r}}\leq\frac{1}{2r}\right\}$. 
 \begin{proposition}\textup{\cite[Theorem A; B]{Per08}}\label{dahlner result}
 	For the weighted Bergman space $\wberg{p}{\alpha}(\udisk)$, $p \geq 1$, $\alpha \geq 0$, the following statements hold:
 	\begin{enumerate}[label=\normalfont(\roman*)]
 	\item  For each $\lambda$ in the interior of $\spec{\ces}{\wberg{p}{\alpha}}$, the set $\imsp{\ces}$ is a closed one codimensional subspace of $\wberg{p}{\alpha}$.
 	\item $\ptsp{\ces}{\wberg{p}{\alpha}}=\left\{\frac{1}{m}: m \in \nn, \, m <\frac{2+\alpha}{p}\right\}$.
 	\item $\spec{\ces}{\wberg{p}{\alpha}}=\overline{D}_{\frac{2+\alpha}{p}}\cup  \ptsp{\ces}{\wberg{p}{\alpha}}$.
 	\end{enumerate}
 \end{proposition}
  To prove one of our main theorems, we need the following abstract spectral result.
\begin{lemma}\textup{\cite[Lemma 2.1]{ABR17}}\label{union of spectra}
	Let $E=\bigcap_{n \in \nn}E_n$ be a Fréchet space which is the intersection of a sequence of Banach spaces $((E_n,\norm{\cdot}_n))_{n \in \nn}$ satisfying $E_{n+1} \subseteq E_n$ with $\norm{x}_n \leq \norm{x}_{n+1}$, for all $n \in \nn$ and $x \in E_{n+1}$. Let $T \in \clo{E}$ satisfy:
	 \begin{enumerate}[label=\normalfont(\Alph*)]
	 	\item For all $n \in \nn$, there exists $T_n \in \clo{E_n}$ such that the restriction of $T_n$ to $E$ (resp. of $T_n$ to $E_{n+1}$) coincides with $T$ (resp. $T_{n+1}$).
	 \end{enumerate}
Then, the following statements hold:
\begin{enumerate}[wide, labelwidth=!, labelindent=0pt, label=\normalfont(\roman*)]\setlength\itemsep{0.5em}
\item $\spec{T}{E} \subseteq \cup_{n \in \nn} \spec{T_n}{E_n}$ and $R(\lambda,T)$ coincides with the restriction of $R(\lambda,T_n)$ to $E$, for all $n \in \nn$ and $\lambda \in \bigcap_{n \in \nn}\rho(T_n;E_n)$.
\item If $\cup_{n \in \nn} \sigma(T_n;E_n)\subseteq \overline{\sigma(T;E)}$, then
\[
\stsp{T}{E} = \overline{\spec{T}{E}}.
\]
\end{enumerate}
\end{lemma}
For $n \in \nn$, let us denote by $\ces_n$ the Cesàro operator acting on the Banach space $\wberg{p}{\alphapp{n}}$.
\begin{theorem}\label{main result}
	Let $1< p<\infty$. Then, for $0 < \alpha<\infty$, the following statements hold:
	\begin{enumerate}[label=\normalfont(\arabic*)]
	\item We have the inclusion 
	\[
	\left\{\frac{1}{m}: m \in \nn, \, m <\frac{2+\alpha}{p}\right\}
 \subset \ptsp{\ces}{\iberg{p}{\alpha}} \subset \left\{\frac{1}{m}: m \in \nn, \, m\leq\frac{2+\alpha}{p}\right\}
	\]
	\item $\spec{\ces}{\iberg{p}{\alpha}} =\{0\}\cup D_{\frac{2+\alpha}{p}}\cup\ptsp{\ces}{\iberg{p}{\alpha}}$.
	\item $\stsp{\ces}{\iberg{p}{\alpha}} =\overline{\spec{\ces}{\iberg{p}{\alpha}}}$.
	\end{enumerate}
\end{theorem}

\begin{proof}
\begin{enumerate}[wide, labelwidth=!, labelindent=0pt, label=\normalfont(\arabic*)]\setlength\itemsep{0.5em}
 \item One inclusion follows from Proposition~\ref{dahlner result}. For the other inclusion, take any $\lambda \in \ptsp{\ces}{\iberg{p}{\alpha}}$. Then, there exists $f \in \iberg{p}{\alpha}$ such that $\ces f=\lambda f$. Since $f \in \wberg{p}{\mu}$ for every $\mu>\alpha$, we have $\ces f=\lambda f$ in $\wberg{p}{\mu}$ as well. Then $\lambda \in \ptsp{\ces}{\wberg{p}{\mu}}$, for all $\mu>\alpha$. Hence $\lambda \in \cap_{\mu>\alpha} \ptsp{\ces}{\wberg{p}{\mu}}$. By Proposition~\ref{dahlner result}(ii) we obtain
 \begin{align*}
 	\ptsp{\ces}{\iberg{p}{\alpha}} & \subseteq \bigcap_{\mu>\alpha} \ptsp{\ces}{\wberg{p}{\mu}} \\
 	& = \bigcap_{\mu>\alpha} \left\{\frac{1}{m}\colon m \in \nn, m<\frac{2+\mu}{p} \right\} \\
 	& = \left\{\frac{1}{m}\colon m \in \nn, m \leq \frac{2+\alpha}{p} \right\}.
 \end{align*}
 \item By part (1) and Lemma~\ref{union of spectra}(i), we already know that
\[
\ptsp{\ces}{\iberg{p}{\alpha}} \subseteq \spec{\ces}{\iberg{p}{\alpha}} \subseteq \bigcup_{n \in \nn} \spec{\ces_n}{\wberg{p}{\alphapp{n}}}.
\]
Proposition~\ref{ces isom on ibergp} implies that $0 \in \spec{\ces}{\iberg{p}{\alpha}}$. By Proposition~\ref{dahlner result}, we obtain
\[
\spec{\ces}{\iberg{p}{\alpha}} \subseteq  \{0\}\cup \bigcup_{n \in \nn}\left(\left\{\frac{1}{m}: m \in \nn, m<\frac{2+\alphapp{n}}{p} \right\}\cup \overline{D}_{\frac{2+\alphapp{n}}{p}} \right).
\]
Then we clearly have
\[
\spec{\ces}{\iberg{p}{\alpha}} \subseteq  \{0\} \cup \left\{\frac{1}{m}: m\in \nn, m \leq \frac{2+\alpha}{p} \right\} \cup D_{\frac{2+\alpha}{p}}.
\]
\normalsize
So it remains to show that $D_{\frac{2+\alpha}{p}} \subseteq \spec{\ces}{\iberg{p}{\alpha}}$. Let us fix $\alpha>0$, and let $\lambda \in D_{\frac{2+\alpha}{p}}$. Let us set $m_\alpha:=\max\{m \in \nn: \frac{1}{m} \geq \frac{p}{2+\alpha}\}$. Then find $n_0 \in \nn$ such that $\frac{p}{2+\alphapp{n}} \geq \frac{1}{m_\alpha+1}$ and since $\lambda \in D_{m_\alpha+1}$ we have 
\begin{align*}
	\snorm{\lambda-\frac{p}{2(2+\alphapp{n})}} & \leq \snorm{\lambda-\frac{1}{2(m_\alpha+1)}} \\
	& < \frac{1}{2(m_\alpha+1)} \\
	& \leq \frac{p}{2(2+\alphapp{n})},
\end{align*}
for every $n\geq n_0$. Hence, $\lambda \in D_{\frac{2+\alphapp{n}}{p}}$, for all $n \geq n_0$. So by Proposition~\ref{dahlner result}(ii), $\lambda$ belongs to the interior of $\spec{\ces_n}{\wberg{p}{\alphapp{n}}}$, for each $n \geq n_0$.  We next show that the set $\iberg{p}{\alpha} \setminus\imsp{\ces}$ is a non-empty open set. Due to Proposition~\ref{dahlner result}(i), the argument is as in the proofs of \cite[Theorem 2.2]{ABR17} and \cite[Proposition 2.3]{ABR18-3}. We first take any sequence $(g_j)_{j \in \nn}\subseteq \imsp{\ces}$ such that $g_j \goesto{j} g \in \iberg{p}{\alpha}$. For every $j \in \nn$ let us select $f_j \in \iberg{p}{\alpha}$ such that $(\lambda I-\ces)f_j=g_j$. In particular, $f_j \subseteq \wberg{p}{\alphapp{n}}$, for every $n \in \nn$. Then we have $g_j \goesto{j} g \in \wberg{p}{\alphapp{n}}$. Since $\imsp{\ces_n}$ is closed in $\wberg{p}{\alphapp{n}}$, $g \in \imsp{\ces_n}$, for all $n \geq n_0$. Then there exists $h_n \in \wberg{p}{\alphapp{n}}$ such that $\imsp{\ces_n} h_n=g$. Moreover, for $n \geq n_0$ we have $\imsp{\ces_n} h_n=g=(\lambda I-\ces_{n+1})h_{n+1}$. Since the restriction of $\ces_n$ to $\wberg{p}{\alphapp{n+1}}$ coincides with $\ces_{n+1}$ and $\lambda I-\ces_n$ is injective, we have $h_n=h_{n+1}$, for all $n \geq n_0$. So $g \in \imsp{\ces}$, and hence $\imsp{\ces}$ is closed. Now it remains to show that $\imsp{\ces}$ is a proper subspace. Assume not, that is, suppose that $\imsp{\ces} =\iberg{p}{\alpha}$. Since $\iberg{p}{\alpha}$ is dense in $\wberg{p}{\alphapp{n}}$, for all $n \in \nn$,
\begin{equation}\label{huge inclusion}
\wberg{p}{\alphapp{n}} =\overline{\iberg{p}{\alpha}}=\overline{(\lambda I-\ces)(\iberg{p}{\alpha})}\subseteq (\lambda I-\ces_n)(\wberg{p}{\alphapp{n}}), 
\end{equation}
where all the closures are taken in $\wberg{p}{\alphapp{n}}$. However, this contradicts the fact that $\imsp{\ces_n}$ is a closed subspace of $\wberg{p}{\alphapp{n}}$. Hence $\lambda I-\ces$ is not surjective, so $\lambda \in \spec{\ces}{\iberg{p}{\alpha}}$. 
\item By part (2), we observe that
\[
\overline{\spec{\ces}{\iberg{p}{\alpha}}} = \{0\} \cup \overline{D}_{\frac{2+\alpha}{p}} \cup \left\{\frac{1}{m}: m\in \nn, m \leq \frac{2+\alpha}{p}\right\}.
\]
From Proposition~\ref{dahlner result}(ii), we deduce that
\[
\bigcup_{n \in \nn} \spec{\ces_n}{\wberg{p}{\alphapp{n}}} \subseteq \{0\} \cup \overline{D}_{\frac{2+\alpha}{p}} \cup \left\{\frac{1}{m}: m\in \nn, m \leq \frac{2+\alpha}{p}\right\}.
\]
So by Lemma~\ref{union of spectra}(ii), we get the desired result. 
 \end{enumerate}
\end{proof}
\subsection{Spectrum of $\ces$ on $\uberg{p}{\alpha}$}\label{subsection: spectrum on u}
Let us state an abstract spectral lemma which is needed for our next result.
\begin{lemma}\label{abstract result ii}\textup{\cite[Lemma 5.2]{ABR18-3}}
	Let $E=\ind_{n \in \nn}(E_n, \norm{\cdot}_n)$ be a Hausdorff inductive limit of Banach spaces. Let $T \in \clo{E}$ satisfy the following condition:
	 \begin{enumerate}[label=\normalfont(\Alph*)]
	 	\item For each $n \in \nn$, the restriction $T_n$ of $T$ to $E_n$ maps $E_n$ into itself and $T_n \in \clo{E_n}$.
	 \end{enumerate}
Then, the following properties are satisfied:
	 \begin{enumerate}[label=\normalfont(\roman*)]
	 	\item $\ptsp{T}{E}=\bigcup_{n \in \nn}\ptsp{T_n}{E_n}$.
	 	\item $\spec{T}{E}\subseteq \bigcap_{m \in \nn}(\bigcup_{n=m}^\infty \spec{T_n}{E_n}$. Moreover, if $\lambda \in \bigcap_{n=m}^\infty \rho(T_n;E_n)$ for some $m \in \nn$, then $R(\lambda,T_n)$ coincides with the restriction of $R(\lambda,T)$ to $E_n$, for every $n\geq m$.
	 	\item If $\bigcup_{n=m}^\infty \spec{T_n}{E_n}\subseteq \overline{\spec{T}{E}}$, for some $m \in \nn$, then 
	 	\[ \stsp{T}{E}=\overline{\spec{T}{E}}.\]
	 \end{enumerate}
\end{lemma}
In the light of Lemma~\ref{abstract result ii}, the following result will follow by using the arguments in \cite[Propositions 2.5 - 2.9]{ABR18-3}, adapted to our setting.
\begin{theorem}\label{main result ii}
	Let $1<p<\infty$ be fixed, and let $0<\alpha<\infty$. Then, the following statements hold:
	\begin{enumerate}[label=\normalfont(\arabic*)]
		\item $\ptsp{\ces}{\uberg{p}{\alpha}} = \{\frac{1}{m}:m \in \nn, m<\frac{2+\alpha}{p}\}$.
		\item $\spec{\ces}{\uberg{p}{\alpha}} = \ptsp{\ces}{\uberg{p}{\alpha}}\cup \overline{D}_{\frac{2+\alpha}{p}}$.
		\item $\stsp{\ces}{\uberg{p}{\alpha}}=\spec{\ces}{\uberg{p}{\alpha}}$.
	\end{enumerate}
\end{theorem}
\begin{proof}
	\begin{enumerate}[wide, labelwidth=!, labelindent=0pt, label=\normalfont(\arabic*)]\setlength\itemsep{0.5em}
	\item By Lemma~\ref{abstract result ii}(i), we know that
		 \begin{align*}
		 \ptsp{\ces}{\uberg{p}{\alpha}} &= \bigcup_{n \in \nn} \ptsp{\ces_n}{\wberg{p}{\alphapm{n}}}\\
		 & = \bigcup_{n\in \nn} \left\{\frac{1}{m}:m\in \nn, m<\frac{2+\alphapm{n}}{p} \right\} \\
		 & = \left\{\frac{1}{m}:m\in \nn, m<\frac{2+\alpha}{p}\right\}.
		\end{align*}
	\item If we apply Lemma~\ref{abstract result ii}(ii) we obtain
	\[
		\spec{\ces}{\uberg{p}{\alpha}} \subseteq \bigcap_{m \in \nn, m>\frac{1}{\alpha}}\left(\bigcup_{n=m}^\infty \spec{\ces_n}{\wberg{p}{\alphapm{n}}} \right).
	\]
	On the other hand, by Proposition~\ref{dahlner result}(ii) we know that for every $n \in \nn$ satisfying $n > \frac{1}{\alpha}$
	\[
	\spec{\ces_n}{\wberg{p}{\alphapm{n}}}=\left\{\frac{1}{m}:m\in \nn, m<\frac{2+\alphapm{n}}{p}\right\}\cup D_{\frac{2+\alphapm{n}}{p}}
	\]
	Since for each $n,m \in \nn$ with $n \geq m>\frac{1}{\alpha}$ one has
	\[
	D_{\frac{2+\alphapm{n}}{p}} \subseteq D_{\frac{2+\alphapm{m}}{p}},
	\]
	it follows
	\[
	\spec{\ces}{\uberg{p}{\alpha}} \subseteq \left\{\frac{1}{m}:m\in \nn, m<\frac{2+\alpha}{p} \right\} \cup \overline{D}_{\frac{2+\alpha}{p}}.
	\]
	For the other inclusion, first let us notice that by part (1), 
	\[
	\left\{\frac{1}{m}:m\in \nn, m<\frac{2+\alpha}{p}\right\} \subseteq \spec{\ces}{\uberg{p}{\alpha}}.
	\]
	Now let us assume that there exists $\lambda \in \cn$ satisfying $\snorm{\lambda-\frac{p}{2(2+\alpha)}}<\frac{p}{2(2+\alpha)}$, but $\lambda \notin \spec{\ces}{\uberg{p}{\alpha}}$. Then, $(\lambda I-\ces)(\uberg{p}{\alpha})=\uberg{p}{\alpha}$. However, Proposition~\ref{dahlner result}(i) implies that $\imsp{\ces}$ is a one-dimensional closed subspace of $\wberg{p}{\alpha}$. Since $\uberg{p}{\alpha}$ is dense in $\wberg{p}{\alpha}$, we have
	\[
	\wberg{p}{\alpha} = \overline{\uberg{p}{\alpha}} = \overline{(\lambda I-\ces)(\uberg{p}{\alpha})} \subseteq (\lambda I-\ces)(\wberg{p}{\alpha}),
	\]
	closures taken in $\wberg{p}{\alpha}$. Then, $\imsp{\ces}$ is dense in $\wberg{p}{\alpha}$. Contradiction. Therefore $\lambda \in \spec{\ces}{\uberg{p}{\alpha}}$. Now it remains to show that
	\[
	\left\{\lambda\in\cn:\snorm{\lambda-\frac{p}{2(2+\alpha)}}=\frac{p}{2(2+\alpha)} \right\} \subseteq \spec{\ces}{\uberg{p}{\alpha}}.
	\]
	Fix $\lambda\in \cn$ such that $\snorm{\lambda-\frac{p}{2(2+\alpha)}}=\frac{p}{2(2+\alpha)}$, that is, $\operatorname{Re}(\frac{1}{\lambda})= \frac{2+\alpha}{p}$. Then, by \cite[Proposition 4]{Per08}, constant functions do not belong to $(\lambda I-\ces)(\wberg{p}{\alpha})$. For instance, take $\boldsymbol{1} \in \uberg{p}{\alpha} \subseteq \wberg{p}{\alpha}$ so that $\boldsymbol{1}\notin (\lambda I-\ces)$ and hence $(\lambda I-\ces)\colon\uberg{p}{\alpha}\to\uberg{p}{\alpha}$ fails to be surjective. Therefore, $\lambda \in \spec{\ces}{\uberg{p}{\alpha}}$, and this completes the proof.
	\item By definition, it is always true to say $\spec{\ces}{\uberg{p}{\alpha}}\subseteq \stsp{\ces}{\uberg{p}{\alpha}}$. For the reverse inclusion, let us take $\lambda \notin \spec{\ces}{\uberg{p}{\alpha}}$ and show $\lambda \notin \stsp{\ces}{\uberg{p}{\alpha}}$. By part (2), the set $\spec{\ces}{\uberg{p}{\alpha}}$ is compact. So there exist $r>0$ and $n_0 \in \nn$ with $n_0>\frac{1}{\alpha}$ such that 	$\overline{B(\lambda,r)}\cap \spec{\ces}{\wberg{p}{\alphapm{n_0}}}=\varnothing$. By Proposition~\ref{dahlner result}(iii), we have $\overline{B(\lambda,r)} \cap \spec{\ces_n}{\wberg{p}{\alphapm{n}}}=\varnothing$, for every $n\geq n_0$. That is, $\overline{B(\lambda,r)}\subseteq \rho(\ces_n;\wberg{p}{\alphapm{n}})$, for every $n\geq n_0$. Hence the set $\{(\mu I-\ces_n)^{-1}\colon\wberg{p}{\alphapm{n}}\to\wberg{p}{\alphapm{n}}|\, \mu \in \overline{B(\lambda,r)}\}$ is equicontinuous, for every $n \geq n_0$. Now we show that the set $\{(\mu I-\ces)^{-1}\colon\uberg{p}{\alpha}\to\uberg{p}{\alpha}|\,\mu \in \overline{B(\lambda,r)}\}$ is equicontinuous in $\clo{\uberg{p}{\alpha}}$. As we shall show that $\uberg{p}{\alpha}$ is a (DFS)-space (see Corollary~\ref{schwartz}), it is barrelled. Hence, by means of Banach-Steinhaus Theorem, it suffices to show that the set
	 \[
	 \{(\mu I-\ces)^{-1}f:\mu\in\overline{B(\lambda,r)}\}
	\]
	is bounded in $\uberg{p}{\alpha}$. This will guarantee $\lambda \notin \stsp{\ces}{\uberg{p}{\alpha}}$. Let us fix $f \in \wberg{p}{\alphapm{n}}$, for some $n\geq n_0$. So $\{(\mu I-\ces_n)^{-1}f\colon \mu \in \overline{B(\lambda,r)}\}$ is a bounded set in $\wberg{p}{\alphapm{n}}$ and hence also in $\uberg{p}{\alpha}$. Since $(\mu I-\ces)^{-1}|_{\wberg{p}{\alphapm{n}}}=(\mu I-\ces_n)^{-1}$ for $\mu \in \overline{B(\lambda,r)}$, we are done. 
	\end{enumerate}
\end{proof}

\subsection{Other properties of $\ces$ on $\iberg{p}{\alpha}$ and $\uberg{p}{\alpha}$}
The following result follows from Section~\ref{subsection: spectrum on i} and Section~\ref{subsection: spectrum on u}.
\begin{proposition}\label{compactness of ces}
	Let $1<p<\infty$ be fixed, and let $0<\alpha<\infty$. Then, the Cesàro operator $\ces$ fails to be compact on $\iberg{p}{\alpha}$ and $\uberg{p}{\alpha}$.
\end{proposition}

\begin{proof}
	We make use of the results in \cite[Theorem 9.10.2(4)]{Edw65} and \cite[Theorem 2.4]{Gro73} stating that a compact operator $T\colon X \to X$ on a Hausdorff locally convex space $X$ necessarily has a spectrum $\spec{T}{X}$ which is compact as a set in $\cn$ and every element of $\spec{T}{X}$ except for the origin is isolated. We see in Theorem~\ref{main result} $\ces$ has a non-compact spectrum on $\iberg{p}{\alpha}$. Although it has a compact spectrum on $\uberg{p}{\alpha}$ as shown in Theorem~\ref{main result ii}, the points in the spectrum are not isolated. 
\end{proof}

We conclude this section with some remarks on the dynamical properties of the Cesàro operator $\ces$ on $\iberg{p}{\alpha}$. A Fréchet space operator $T \in \clo{X}$, where $X$ is separable, is called \textit{hypercyclic} if there exists $x \in X$ such that the orbit $\{T^kx:k \in \nn_0\}$ is dense in $X$. If, for some $z \in X$, the projective orbit $\{\lambda T^kz:\lambda \in \cn, k \in \nn_0\}$ is dense in $X$, then $T$ is called \textit{supercyclic}. Clearly, if $T$ is hypercyclic then $T$ is supercyclic.

\begin{proposition}
	Let $1<p<\infty$, and $0<\alpha<\infty$. Then, both $\ces\colon\iberg{p}{\alpha}\to \iberg{p}{\alpha}$ and $\ces\colon\uberg{p}{\alpha}\to\uberg{p}{\alpha}$ fail to be supercyclic. In particular, they are not hypercyclic.
\end{proposition}

\begin{proof}
	It is proved in \cite[Proposition 2.20]{ABR18} that the Cesàro operator $\ces$ acting on $\holom$ fails to be supercyclic. So if $\ces\colon \uberg{p}{\alpha}\to \uberg{p}{\alpha}$ or $\ces\colon\iberg{p}{\alpha}\to \iberg{p}{\alpha}$ were supercyclic, by the fact that $\iberg{p}{\alpha}$ and $\uberg{p}{\alpha}$ are dense in $\holom$ since it contains the polynomials, $\ces$ would also be supercyclic on $\holom$. This is a contradiction. 
\end{proof}

\section{Further discussion on the structures of $\iberg{p}{\alpha}$ and $\uberg{p}{\alpha}$}\label{section: discussion on iberg}
The well-known \textit{Korenblum space} \cite{Kor75} is defined by 
\[
 A^{-\infty}:=\bigcup_{0<\alpha<\infty} A^{-\alpha} = \bigcup_{n\in \nn} A^{-n},
\]
where 
\[
A^{-\alpha}=\{f \in \holom\colon \supx{z}{\udisk}(1-\snorm{z})^\alpha \snorm{f(z)}<\infty\}.
\]
The classical Korenblum space $A^{-\infty}$ is a regular (LB)-space when endowed with the finest locally convex topology which makes each natural inclusion map $A^{-n}\subseteq A^{-\infty}$ continuous. So, $A^{-\infty}=\indx{n}{\nn}A^{-n}$. Let $f \in \wberg{p}{\alpha}$. Then, by means of \cite[Lemma 3.1]{Per08}, we obtain $f \in A^{-(\frac{2+\alpha}{p})}_0$, where 
 \[
 A^{-\alpha}_0=\{f \in \holom\colon \limx{\snorm{z}}{1^-}(1-\snorm{z})^\alpha \snorm{f(z)}=0\} \subseteq A^{-\alpha}.
 \]
 This means $\cup_{\alpha \in \nn} \wberg{p}{\alpha} \subset A^{-\infty}$. For the converse, let us take any $f \in A^{-\infty}$. Then there exists $n \in \nn$ and a constant $M>0$ such that $\supx{z}{\udisk} (1-\snorm{z})^n \snorm{f(z)}<M$. Then,
\[
\dint \snorm{f(z)}^p (1-\snorm{z})^{np} \difs{z} \leq M^p \dint \difs{z} = M^p,
\]
which implies that $f \in \wberg{p}{np}$. This shows $A^{-\infty}\subset \cup_{\alpha>0} \wberg{p}{\alpha}$. See also \cite[p. 111]{HKZ00}. Continuity, compactness and the spectrum of the Cesàro operator acting on the classical Korenblum space have been studied completely in \cite{ABR18-3}. This was the reason we avoided any study of $\ces$ acting on the (LB)-space $\cup_{\alpha>0} \wberg{p}{\alpha}$ although it seems quite tempting concerning the nature of the spaces we deal with.
\subsection{Schwartz property in $\iberg{p}{\alpha}$ and $\uberg{p}{\alpha}$}
For each $0<\mu<\gamma$ it is easy to observe that it holds for the pair of weighted Bergman spaces $\wberg{p}{\mu} \subseteq \wberg{p}{\gamma}$.
\begin{lemma}\label{linking map compact}
	For each pair $0<\mu<\gamma<\infty$, the canonical inclusion map $\iota:\wberg{p}{\mu} \hookrightarrow \wberg{p}{\gamma}$ is compact.
\end{lemma}

\begin{proof}
Fix $0<\mu<\gamma<\infty$, $M>0$. By the fact that $\lim_{\snorm{z} \to 1^-}(1-\snorm{z})^{\gamma-\mu}=0$, for any $\varepsilon>0$ given, we find $R \in (0,1)$ such that $(1-\snorm{z})^{\gamma-\mu}<\frac{\varepsilon}{2M^p}$, for all $\snorm{z}>R$. Let us take $f=(f_j) \in \wberg{p}{\mu}$ with $\wbnorm{f}{p}{\mu} \leq M$, which converges to $0$ in the topology of uniform convergence on compact subsets of $\udisk$. Then, we have
\begin{align*}
	\int\limits_{\snorm{z}>R}\snorm{f_j(z)}^p\diff{\gamma} & = \int\limits_{\snorm{z} > R} \snorm{f_j(z)}^p (1-\snorm{z})^{\gamma-\mu} \diff{\mu} \\
& \leq \frac{\varepsilon}{2M^p}\int\limits_{\snorm{z} > R} \snorm{f_j(z)}^p \diff{\mu}  \leq \frac{\varepsilon}{2M^p} \wbnorm{f}{p}{\mu}^p  \leq \frac{\varepsilon}{2}.
\end{align*}
Now let us find $j_0 \in \nn$ such that we have $\snorm{f_j(z)}<(\frac{\varepsilon}{2(\gamma+1)})^{1/p}$, for all $j \geq j_0$. So, for every $j\geq j_0$ we obtain
\[
	\int\limits_{\snorm{z} \leq R} \snorm{f_j(z)}^p \diff{\gamma} \leq \frac{\varepsilon}{2(\gamma+1)} \int\limits_{\snorm{z} \leq R} \diff{\gamma} = \frac{\varepsilon}{2}.
\]
Therefore, combining the arguments above, we obtain $\wbnorm{\iota(f_j)}{p}{\gamma}^p \leq \varepsilon$. This means, $f_j$ converges to $0$ in norm topology as well. Then, $\iota$ is a compact operator. 
\end{proof}
\begin{corollary}\label{schwartz}
 For a fixed $1<p<\infty$ and for all $0<\alpha<\infty$. Then,
 \begin{enumerate}[label=\normalfont (\arabic*)]
 	\item $\iberg{p}{\alpha}$ is a Fréchet-Schwartz space.
 	\item $\uberg{p}{\alpha}$ is a (DFS)-space.
 \end{enumerate}
\end{corollary}

\begin{proof}
\begin{enumerate}[wide, labelwidth=!, labelindent=0pt, label=\normalfont(\arabic*)]\setlength\itemsep{0.5em}
	\item By \eqref{projlim verion}, this result follows directly by the combination of arguments in Lemma~\ref{linking map compact}, and \cite[\S 21.1, Example 1(b)]{Jar81}.
	\item By \eqref{injlim verion}, it follows by Lemma~\ref{linking map compact} and \cite[Proposition 25.20]{Vog97}. 
	\end{enumerate}
\end{proof}

\subsection{On nuclearity of $\iberg{p}{\alpha}$ and $\uberg{p}{\alpha}$}
A sequence $(x_j)_{j=0}^\infty$ in a locally convex space $E$ is said to be a \textit{Schauder basis} if each element $y \in E$ can be written uniquely as $y=\sum_{j=1}^\infty f_j(y)x_j$, where $f_j\colon E \to \mathbb{K}$, $j \in \nn$ are continuous linear forms. See \cite{Jar81} for further information on Schauder basis in Fréchet spaces. The work of Lusky \cite[Theorem 2.2]{Lus00} tells us that the monomials $\varLambda=\{z^j\}_{j=0}^\infty$ is a Schauder basis for $\wberg{p}{\alpha}$, since it is proved for a more general setup (see also \cite{BLT19}). In the light of that, and with the help of Fréchet and inductive limit topologies, it is straightforward to assert that
\begin{theorem}\label{monomial basis}
	Let $1<p<\infty$ be fixed, and let $0<\alpha<\infty.$ Then, 
	\begin{enumerate}[label=\normalfont(\arabic*)]
		\item $\varLambda$ is a Schauder basis for $\iberg{p}{\alpha}$.
		\item $\varLambda$ is a Schauder basis for $\uberg{p}{\alpha}$.
	\end{enumerate}
\end{theorem}

\begin{proof}
\begin{enumerate}[wide, labelwidth=!, labelindent=0pt, label=\normalfont(\arabic*)]\setlength\itemsep{0.5em}
\item We need to prove that the Taylor series of any $f \in \iberg{p}{\alpha}$, $\alpha\geq 0$, converges in $\iberg{p}{\alpha}$ to $f$.  Let us fix $\mu>\alpha$ and pick $\mu_1$ with $\alpha<\mu_1<\mu$. Since $f \in \wberg{p}{\mu_1}$ we may apply \cite[Theorem 2.2]{Lus00} to deduce that the Taylor series of $f$ converges to $f$ in $\iberg{p}{\alpha}$. This implies $\varLambda$ is a Schauder basis for $\iberg{p}{\alpha}$.
\item A direct consequence of \cite[Theorem 2.2]{Lus00} and the properties of inductive limits. 
\end{enumerate}
\end{proof}
Thanks to Theorem~\ref{monomial basis}, we are now allowed to make use of Grothendieck-Pietsch criterion to determine whether $\iberg{p}{\alpha}$ and $\uberg{p}{\alpha}$ are nuclear or not. We adapt this approach from \cite{BLT18}. First we need the following estimate, which is essentially known (cf. \cite[Lemma 4]{SW71}). 

\begin{lemma}\label{monomial estimate}
	Let $1<p<\infty$, and $0<\alpha<\infty$. For $j \in \nn$,
	\[
		\wbnorm{z^j}{p}{\alpha} \simeq \left(\frac{1}{j^{\alpha+1}}\right)^{1/p}.
 	\]
\end{lemma}

\begin{proof}
	For any $f \in \wberg{p}{\alpha}$ we have $\wbnorm{f}{p}{\alpha}^p = \int_0^1 r(1-r)^\alpha\left(\frac{1}{2\pi}\int_0^{2\pi}\snorm{f(re^{it})}^p \dif t \right)\dif r$. Let $\beta(\cdot,\cdot)$ denote the usual Beta function, and $\varGamma(\cdot)$ denote the usual Gamma function. Then we have the estimate
\begin{align*}
	\wbnorm{z^j}{p}{\alpha}^p & \simeq \int_0^1 (1-r)^\alpha\left(\frac{1}{2\pi}\int_0^{2\pi}(r^j)^p \dif t \right)\dif r  = \int_0^1 (1-r)^\alpha r^{jp}\dif r\\
	& = \beta(jp+1,\alpha+1) = \frac{\Gamma(jp+1)\Gamma(\alpha+1)}{\Gamma(jp+\alpha+2)} \\
	& \simeq \frac{(jp)^{jp+\frac{1}{2}}\alpha^{\alphapp{2}}e^{-(\alpha+jp)}}{(jp+\alpha+1)^{jp+\alpha+\frac{3}{2}}e^{-(jp+\alpha+1)}}\\
	& = \left(\frac{jp}{jp+\alpha+1}\right)^{jp+\frac{1}{2}}\frac{\alpha^\alphapp{2}}{e(jp+\alpha+1)} \simeq \frac{1}{j^{\alpha+1}},
\end{align*}
where the first estimate is due to an extension of Stirling's formula (see e.g. \cite[p. 257]{AS72}). 
\end{proof}

\begin{theorem}
	Let $1<p<\infty$ be fixed and let $0<\alpha<\infty$. Then,
	\begin{enumerate}[label=\normalfont(\arabic*)]
		\item The Fréchet-Schwartz space $\iberg{p}{\alpha}$ fails to be nuclear.
		\item The (DFS)-space $\uberg{p}{\alpha}$ fails to be nuclear.
	\end{enumerate}
\end{theorem}

\begin{proof}
	We give the argument for part (1). Suppose that $\iberg{p}{\alpha}$ is nuclear. Then, by Grothendieck-Pietsch criterion (see e.g. \cite[Proposition 28.15]{Vog97}), given $n=1$ we find $m>1$ such that 
	\[
	\sumx{j}{1}{\infty} \frac{\wbnorm{z^j}{p}{\alpha+1}}{\wbnorm{z^j}{p}{\alpha+\frac{1}{m}}}<\infty.
	\]
	On the other hand, by Lemma~\ref{monomial estimate} we have
	\begin{align*}
		\sumx{j}{1}{\infty} \frac{\wbnorm{z^j}{p}{\alpha+1}}{\wbnorm{z^j}{p}{\alpha+\frac{1}{m}}} & \simeq \left(\sumx{j}{1}{\infty} \frac{j^{\alpha+\frac{1}{m}+1}}{j^{\alpha+2}}\right)^{1/p} \\
		& = \left(\sumx{j}{1}{\infty} \frac{1}{j^{1-\frac{1}{m}}}\right)^{1/p} \\
		& = \infty.
	\end{align*}
	This is a contradiction. 
\end{proof}

\section*{Acknowledgements}
This article was completed during the autor's stay at Instituto Universitario de Matemática Pura y Aplicada, Universitat Politècnica de València. The author is indepted to Prof. José Bonet for his helpful suggestions, and for the kind hospitality.
\section*{Declarations}

\textbf{Funding} This research is funded by The Scientific and Technological Research Council of Turkey International Postdoctoral Research Fellowship (TÜBİTAK 2219) with grant number 1059B191800828. \\

\bibliographystyle{spmpsci}
\bibliography{ccesaro}
\end{document}